\pdfoutput=1
\documentclass[draft=false, leqno]{article}
\usepackage{amsmath,amsthm,mathrsfs}
\usepackage[normal]{boilerart1}
\usepackage{longtable}

\theoremstyle{definition}

\newtheorem*{ack*}{Acknowledgments}

\hypersetup{
    urlcolor = slate,
}

\newcommand{\invert}{H}

\newcommand{\period}{\rlap{\ .}}
\newcommand{\comma}{\rlap{\ ,}}

\newcommand{\textbfit}[1]{\textbf{\textit{#1}}}
\newcommand{\upperth}{\textsuperscript{th}\,}

\newcommand{\fupperstar}{f^{\ast}}

\newcommand{\piupperstar}{\pi^{\ast}}

\newcommand{\Gammalowerstar}{\Gamma_{\ast}}
\newcommand{\Gammaupperstar}{\Gamma^{\ast}}

\newcommand{\Gammalowershriek}{\Gamma_{!}}

\newcommand{\sep}{\ensuremath{\textit{sep}}}

\newcommand{\hyp}{\ensuremath{\textit{hyp}}}

\newcommand{\spec}{\ensuremath{\textit{spec}}}

\newcommand{\bc}{\ensuremath{\textit{bc}}}

\newcommand{\of}{\circ}
\newcommand{\equivalent}{\simeq}

\newcommand{\paren}[1]{\left(#1\right)}

\def\reflectedop#1#2{\mathop{\reflectbox{$#1{#2}$}}}

\newcommand{\Topcohbdd}{\Top_{\infty}^{\bc}}
\newcommand{\Topbc}{\Topcohbdd}

\newcommand{\StrTop}{\categ{StrTop}}
\newcommand{\StrTopspec}{\StrTop^{\spec}}

\renewcommand{\Space}{\categ{Spc}}

\newcommand{\Str}{\categ{Str}}

\DeclareMathOperator{\Gal}{Gal}

\DeclareMathOperator{\Pro}{Pro}

\usepackage{todonotes}

\title{Extended étale homotopy groups from profinite Galois categories}

\author{Peter J. Haine}

\date{\today}

\begin{document}

\maketitle

\begin{abstract} 
	In this note we show that the protruncated shape of a spectral $ \infty $-topos is a delocalization of its \textit{profinite} stratified shape.
	This gives a way to reconstruct the extended étale homotopy groups (i.e., the \textit{non}-profin\-ite\-ly complete étale homotopy groups) of a coherent scheme from its \textit{profinite} Galois category.
\end{abstract}

\setcounter{tocdepth}{2}
\tableofcontents


\section*{Introduction}\addcontentsline{toc}{section}{Introduction}

Let $ X $ be a coherent (i.e., quasicompact quasiseparated) scheme.
In recent work with Clark Barwick and Saul Glasman \cite{exodromy}, we constructed a \textit{delocalization} of the profinite completion of the Artin--Mazur--Friedlander étale homotopy type of $ X $ \cites{MR0245577}{MR676809}.
We call this delocalization the \textit{profinite Galois category} $ \Gal(X) $ of $ X $.
The profinite Galois category $ \Gal(X) $ is pro-object in finite categories, or, equivalently, a category object in profinite topological spaces \cites{Barwick:galperf}[p. 5 \& Construction 13.5]{exodromy}.
The underling category of $ \Gal(X) $ has objects geometric points of $ X $ and morphisms specalizations \textit{in the étale topology} (i.e., is the category of points of the étale topos of $ X $).
Concretely, given geometric points $ \fromto{x}{X} $ and $ \fromto{y}{X} $, a morphism $ \fromto{x}{y} $ in $ \Gal(X) $ is a lift $ \fromto{y}{X_{(x)}} $ of the geometric point $ \fromto{y}{X} $ to the strict localization $ X_{(x)} $ of $ X $ at $ x $.
The topology on $ \Gal(X) $ globalizes the profinite topology on the absolute Galois group $ \Gal(\kappa(x_0)^{\sep}/\kappa(x_0)) $ of the residue field $ \kappa(x_0) $ at each point $ x_0 \in X $.  

From the profinite category $ \Gal(X) $ we can extract a prospace $ \invert(\Gal(X)) $ by formally inverting all morphisms.
Our delocalization result \cite[Examples 11.6 \& 13.6]{exodromy} says that $ \invert(\Gal(X)) $ and the étale homotopy type of $ X $ become (canonically) equivalent after profinite completion.
In this note we provide a stronger relationship between the prospace $ \invert(\Gal(X)) $ and the étale homotopy type: they agree up to \textit{protruncation}.
Morphisms in the $ \infty $-category $ \Pro(\Space) $ of prospaces that induce equivalences after protruncation are precisely those morphisms that become \textit{$ \natural $-isomorphisms} in the category $ \Pro(h\Space) $, in the terminology of Artin--Mazur \cite[Definition 4.2]{MR0245577}.

\begin{mainthm}\label{thm:mainAG}
	Let $ X $ be a coherent scheme and write $ \Pi_{\infty}^{\et}(X) \in \Pro(\Space) $ for the étale homotopy type of $ X $.
	Then there is a natural natural map of prospaces
	\begin{equation*}
		\theta_{X} \colon \fromto{\Pi_{\infty}^{\et}(X)}{\invert(\Gal(X))} \period
	\end{equation*}
	Moreover, $ \theta_{X} $ induces an equivalence on protruncations.
	As a consequence:
	\begin{itemize} 
		\item For each integer $ n \geq 1 $ and geometric point $ \fromto{x}{X} $, we have canonical isomorphisms of progroups
		\begin{equation*}
			\isomto{\pi_{n}^{\et}(X,x)}{\pi_{n}(\invert(\Gal(X)),x)} \comma
		\end{equation*}
		where $ \pi_{n}^{\et}(X,x) $ is the $ n $\upperth homotopy progroup of the étale homotopy type of $ X $.

		\item For any ring $ R $, there is an equivalence of $ \infty $-categories between local systems of $ R $-modules on $ X $ that are uniformly bounded both below and above and continuous functors $ \fromto{\Gal(X)}{D^b(R)} $ that carry every morphism to an equivalence.
	\end{itemize}
\end{mainthm}

The progroups $ \pi_{n}^{\et}(X,x) $ are what we call the \textit{extended étale homotopy groups} of $ X $.
Note that the progroup $ \pi_1^{\et}(X,x) $ is the \textit{groupe fondamentale élargi} of \cite[Exposé X, \S 6]{MR43:223b}; the usual étale fundamenal group of \cite[Exposé V, \S 7]{MR50:7129} is the profinite completion of $ \pi_1^{\et}(X,x) $.


While the protruncated étale homotopy type of a connected Noetherian geometrically unibranch scheme is already profinite \cites[Theorem 11.1]{MR0245577}[Theorem 7.3]{MR676809}[Theorem 3.6.5]{DAGXIII}, in general \Cref{thm:mainAG} provides more refined information about the étale homotopy type, as illustrated in the following example.

\begin{mainexm}
	Consider the nodal cubic curve
	\begin{equation*}
		C = \Spec(\CC[x,y]/(y^2-x^2(x+1))
	\end{equation*}
	over the complex numbers.
	The Riemann Existence Theorem \cites[Theorem 12.9]{MR0245577}[Proposition 4.12]{Carchedi:higheretale}[Theorem 8.6]{MR676809} implies that the \textit{profinite completion} of the étale homotopy type of $ C $ is equivalent to the profinite completion of the circle $ S^1 $.
	It is well-known that, in fact, the \textit{protruncation} of the étale homotopy type of $ C $ is $ S^1 $; \Cref{thm:mainAG} provides an easy `categorical' explanation of this fact.

	There is a continuous functor from $ \Gal(C) $ to the poset category $ \{0 < 1\} $ given by sending the node point to $ 0 $ and every other geometric point to $ 1 $.
	The local ring $ O_{C,(x,y)} $ at the node point has two prime ideals and the strict Henselization of $ O_{C,(x,y)} $ is isomorphic to the strict Henselization of 
	\begin{equation*}
		(\CC[u,v]/(uv))_{(u,v)} \period
	\end{equation*} 
	Using this one sees that there are two lifts of the generic geometric point of $ C $ to the strict localization of $ C $ at the node.
	Hence the continuous functor $ \fromto{\Gal(C)}{\{0 < 1\}} $ factors through the category $ D $ with two objects $ 0 $ and $ 1 $ and two distinct morphisms $ \parto{0}{1} $.
	Moreover, the functor $ \fromto{\Gal(C)}{D} $ induces an equivalence on underlying homotopy types: the prospace $ \invert(\Gal(C)) $ is equivalent to $ \invert(D) \equivalent S^1 $.
	\Cref{thm:mainAG} now shows that the protruncation of the étale homotopy type of the nodal cubic is $ S^1 $.
\end{mainexm}

We relate the étale homotopy type and profinite Galois category of a coherent scheme by situating the problem in a more general context. 
In \cite{exodromy} we provided an equivalence of $ \infty $-categories 
\begin{equation*}\label{eq:Hochsterequiv}
	\widetilde{(-)} \colon \equivto{\Pro(\Str_{\pi})}{\StrTopspec_{\infty}}
\end{equation*}
between the \textit{$ \infty $-category of profinite stratified spaces} (on the left) and the \textit{$ \infty $-category of spectral stratified $ \infty $-topoi} (on the right) \cite[Theorem 10.10]{exodromy}.
The primary example of a spectral stratified $ \infty $-topos is the étale $ \infty $-topos $ X_{\et} $ of a coherent scheme $ X $ with its natural stratification by the Zariski space of $ X $ \cite[Example 10.6]{exodromy}.
The corresponding profinite stratified space is the profinite Galois category $ \Gal(X) $ \cite[Construction 13.5]{exodromy}.

The equivalence $ \Pro(\Str_{\pi}) \equivalent \StrTopspec_{\infty} $ provides a way to reconstruct the prospace given by the shape of the étale $ \infty $-topos of a coherent scheme $ X $\footnote{This is, up to protruncation, the Artin--Mazur--Friedlander étale homotopy type of $ X $; see \cite[\S 5]{MR3763287}, which we recall in \Cref{exm:etalehomotopy,exm:twoetaletypes}.} from its profinite Galois category $ \Gal(X) $, via the composite
\begin{equation*}
	\begin{tikzcd}[sep=1.5em]
		\Pro(\Str_{\pi}) \arrow[r, "\sim"{yshift=-0.25em}] & \StrTopspec_{\infty} \arrow[r] & \Top_{\infty} \arrow[r, "\Pi_{\infty}"] & \Pro(\Space) \comma
	\end{tikzcd}
\end{equation*}
where the middle functor functor forgets the stratification, and $ \Pi_{\infty} $ is the shape (see \Cref{dfn:shape}).
There's another functor $ \invert \colon \fromto{\Pro(\Str_{\pi})}{\Pro(\Space)} $ that doesn't require the use of $ \infty $-topoi, namely, the extension to pro-objects of the composite 
\begin{equation*}
	\begin{tikzcd}[sep=1.5em]
		\Str_{\pi} \arrow[r] & \Cat_{\infty} \arrow[r, "\invert"] & \Space \comma
	\end{tikzcd}
\end{equation*}
where the first functor forgets the stratification and the second functor sends an $ \infty $-category $ C $ to the homotopy type $ H(C) $ obtained by inverting every morphism in $ C $.
It follows formally that these two functors agree on $ \Str_{\pi} $.
Moreover, as the extension to pro-objects of a functor $ \fromto{\Str_{\pi}}{\Space} $, the functor $ \invert \colon \fromto{\Pro(\Str_{\pi})}{\Pro(\Space)} $ preserves inverse limits. 
Thus we have a map
\begin{equation*}
	\theta_C \colon \fromto{\Pi_{\infty}(\widetilde{C})}{\invert(C)}
\end{equation*}
natural in $ C \in \Pro(\Str_{\pi}) $.
In this note we prove that this map is an equivalence after protruncation:

\begin{mainthm}[\Cref{thm:mainthm}]\label{thm:main}
	Let $ \Space_{<\infty} \subset \Space $ denote the $ \infty $-category of truncated spaces, and write $ \tau_{<\infty} \colon \fromto{\Pro(\Space)}{\Pro(\Space_{<\infty})} $ for the left adjoint to the inclusion.
	For any profinite stratified space $ C $, the natural map
	\begin{equation*}
		\tau_{<\infty} \theta_C \colon \fromto{\tau_{<\infty}\Pi_{\infty}(\widetilde{C})}{\tau_{<\infty}\invert(C)}
	\end{equation*}
	of protruncated spaces is an equivalence.
\end{mainthm}

\noindent In light of \cite[Construction 13.5]{exodromy}, \Cref{thm:mainAG} is immediate from \Cref{thm:main}. 

Since the functor $ \invert $ and the shape $ \Pi_{\infty} $ agree on $ \Str_{\pi} $ and both $ \invert $ and $ \tau_{<\infty} $ preserve inverse limits, by the universal property of the $ \infty $-category of pro-objects, \Cref{thm:main} follows once we know that the the protruncated shape $ \tau_{<\infty} \Pi_{\infty} $ preserves inverse limits.
The forgetful functor \smash{$ \fromto{\StrTopspec_{\infty}}{\Top_{\infty}} $} factors through the subcategory \smash{$ \Topbc \subset \Top_{\infty} $} of bounded coherent $ \infty $-topoi and coherent geometric morphisms.
\Cref{thm:main} thus reduces to the following fact.

\begin{mainthm}[\Cref{prop:protruncshapeinverselim}]\label{thm:mainreduction}
	The protruncated shape 
	\begin{equation*}
		\tau_{<\infty} \Pi_{\infty} \colon \fromto{\Topbc}{\Pro(\Space_{<\infty})}
	\end{equation*}
	preserves inverse limits.
\end{mainthm}


In \cref{sec:review} we review the necessary background on pro-objects and shape theory.
The familiar reader should skip straight to \cref{sec:work} where we prove \Cref{thm:main,thm:mainreduction}.

\begin{ack*}
	We thank Clark Barwick for his guidance and sharing his many insights about this material.
	We also gratefully acknowledge support from both the \textsc{mit} Dean of Science Fellowship and \textsc{nsf} Graduate Research Fellowship.
\end{ack*}

\section{Preliminaries on shapes \& protruncated spaces}\label{sec:review}

In this section we review $ \infty $-categories of pro-objects and shape theory for $ \infty $-topoi.
We then record some facts about protruncations that we'll need.


\subsection{Review of shape theory}

\begin{nul}
	We say that a small $ \infty $-category $ I $ is \textbfit{inverse} if the opposite $ \infty $-category $ I^{\op} $ is filtered.
	An \textbfit{inverse system} in an $ \infty $-category $ C $ is a functor $ \fromto{I}{C} $, where $ I $ is an inverse $ \infty $-category.
	An \textbfit{inverse limit} is a limit of an inverse system.

	Let $ C $ be an $ \infty $-category.
	We write $ \Pro(C) $ for the $ \infty $-category of \textbfit{pro-objects} in $ C $ obtained by freely adjoining inverse limits to $ C $, and $ \yo \colon \fromto{C}{\Pro(C)} $ for the Yoneda embedding.
	We say that a pro-object $ X \in \Pro(C) $ is \textbfit{constant} if $ X $ lies in the essential image of $ \yo \colon \fromto{C}{\Pro(C)} $.
	If $ X \colon \fromto{I}{C} $ is an inverse system, we write $ \{X_{i}\}_{i \in I} \coloneq \lim_{i \in I} j(X_i) $ for the pro-object it defines.

	If $ C $ is accessible and admits finite limits, then $ \Pro(C) $ is equivalent to the full subcategory of $ \Fun(C,\Space)^{\op} $ spanned by the left exact accessible functors \SAG{Proposition}{A.8.1.6}.
	Let $ f \colon \fromto{C}{D} $ be a left exact accessible functor between accessible $ \infty $-categories which admit small limits.
	Then the functor $ f \colon \fromto{\Pro(C)}{\Pro(D)} $ admits a left adjoint $ L \colon \fromto{\Pro(D)}{\Pro(C)} $ \SAG{Example}{A.8.1.8}.
	We refer to $ L \of \yo \colon \fromto{D}{\Pro(C)} $ as the \textbfit{pro-left adjoint} of $ f $.
\end{nul}

\begin{ntn}
	We write $ \Cat_{\infty} $ for the $ \infty $-category of $ \infty $-categories and $ \Space \subset \Cat_{\infty} $ for the full subcategory spanned by the $ \infty $-groupoids, i.e., the $ \infty $-category of spaces.

	We write $ \Top_{\infty} \subset \Cat_{\infty} $ for the $ \infty $-category of $ \infty $-topoi and geometric morphisms.
	For any $\infty$-topos $\XX$, we write $ \Gamma_{\XX,\ast} $ or $ \Gamma_{\ast} $ for the global sections geometric morphism, which is the essentially unique geometric morphism $ \fromto{\XX}{\Space} $.
\end{ntn}

\begin{dfn}\label{dfn:shape}
	The \textbfit{shape} $ \Pi_{\infty} \colon \fromto{\Top_{\infty}}{\Pro(\Space)} $ is the left adjoint to the extension to pro-objects of the fully faithful functor $ \incto{\Space}{\Top_{\infty}} $ given by $ \goesto{K}{\Fun(K,\Space)} $ \cite[\SAGsubsec{E.2.2}]{SAG}.
	The shape admits two other very useful descriptions:
	\begin{itemize}
		\item Let $ \XX $ be an $ \infty $-topos, and write $ \Gamma_{!} \colon \fromto{\XX}{\Pro(\Space)} $ for the pro-left adjoint of $ \Gammaupperstar \colon \fromto{\Space}{\XX} $.
		The shape of $ \XX $ is equivalent to the prospace $ \Gamma_{!}(1) $, where $ 1 \in \XX $ denotes the terminal object \cites[\HAappthm{Remark}{A.1.10}]{HA}[\S 2]{MR3763287}.
	
		\item As a left exact accessible functor $ \fromto{\Space}{\Space} $, the prospace $ \Pi_{\infty}(\XX) $ is the composite $ \Gammalowerstar \Gammaupperstar $ \cites[\HTTsec{7.1.6}]{HTT}[\S 2]{MR3763287}.
	\end{itemize}
\end{dfn}

\begin{ntn}
	We write $ \invert \colon \fromto{\Cat_{\infty}}{\Space} $ for the left adjoint to the inclusion.
	The $ \infty $-groupoid $ \invert(C) $ is given by the colimit $ \invert(C) \equivalent \colim_{C} 1_{\Space} $ of the constant diagram $ \fromto{C}{\Space} $ at the terminal object $ 1_{\Space} \in \Space $.
\end{ntn}

\begin{exm}\label{exm:shapeofpresheaf}
	If $ C $ is a small $ \infty $-category, then $ \Gammaupperstar \colon \fromto{\Space}{\Fun(C,\Space)} $ admits a genuine left adjoint $ \Gammalowershriek \colon \fromto{\Fun(C,\Space)}{\Space} $ given by taking the colimit of a diagram $ \fromto{C}{\Space} $.
	The shape of the $ \infty $-topos $ \Fun(C,\Space) $ is thus given by the colimit of the constant diagram at the terminal object of $ \Space $:
	\begin{equation*}
		\Pi_{\infty}(\Fun(C,\Space)) = \Gammalowershriek(1_{\Fun(C,\Space)}) = \textstyle\colim_{C} 1_{\Space} \equivalent \invert(C) \period
	\end{equation*}

	Moreover, the functor $ \invert \colon \fromto{\Cat_{\infty}}{\Space} $ is equivalent to the composite 
	\begin{equation*}
		\begin{tikzcd}[sep=1.5em]
			\Cat_{\infty} \arrow[rrr, "{\Fun(-,\Space)}"] & & & \Top_{\infty} \arrow[r, "\Pi_{\infty}"] & \Space \period
		\end{tikzcd}
	\end{equation*}
\end{exm}

\begin{exm}[{\cite[Corollary 5.6]{MR3763287}}]\label{exm:etalehomotopy}
	If $ X $ is a locally Noetherian scheme, then the Artin--Mazur--Friedlander étale homotopy type of $ X $ corepresents the shape of the \textit{hypercomplete}\footnote{See \cite[\HTTsubsec{6.5.2}]{HTT} for a treatment of hypercomplete $ \infty $-topoi.} étale $ \infty $-topos \smash{$ X_{\et}^{\hyp} $} of $ X $.
	
	The shape of the étale $ \infty $-topos $ X_{\et} $ of $ X $ agrees with the Artin--Mazur-Friedlander étale homotopy type up to \textit{protruncation} (\Cref{exm:twoetaletypes}), to which we now turn.
\end{exm}


\subsection{Protruncated objects}

In this subsection, we recall some facts about protruncated objects and record an interesting observation (\Cref{lem:protruncff}) that we couldn't locate in the literature.

\begin{ntn}
	Let $ C $ be a presentable $ \infty $-category.
	For each integer $ n \geq -2 $, write $ C_{\leq n} \subset C $ for the full subcategory spanned by the $ n $-truncated objects, and $ \tau_{\leq n} \colon \fromto{C}{C_{\leq n}} $ for the $ n $-truncation functor, which is left adjoint to the inclusion $ C_{\leq n} \subset C $ \HTT{Proposition}{5.5.6.18}.
	Write $ C_{<\infty} \subset C $ for the full subcategory spanned by those objects which are $ n $-truncated for some integer $ n \geq -2 $.

	The \textbfit{pro-$ n $-truncation} functor $ \tau_{\leq n} \colon \fromto{\Pro(C)}{\Pro(C_{\leq n})} $ is the extension of the $ n $-truncation functor $ \tau_{\leq n} \colon \fromto{C}{C_{\leq n}} $ to pro-objects.
\end{ntn}

\begin{nul}\label{nul:protrun}
	Let $ C $ be a presentable $ \infty $-category.
	Then the extension to pro-objects of the functor $ \fromto{C}{\Pro(C_{<\infty})} $ given by sending an object $ X \in C $ to the inverse system given by its Postnikov tower $ \{\tau_{\leq n}(X)\}_{n\geq-2} $ is left adjoint to the inclusion $ \incto{\Pro(C_{<\infty})}{\Pro(C)} $.
	We call this left adjoint $ \tau_{<\infty} \colon \fromto{\Pro(C)}{\Pro(C_{<\infty})}$ \textbfit{protruncation}.
	
	A morphism of pro-objects $ f \colon \fromto{X}{Y} $, regarded as left exact accessible functors $ \fromto{C}{\Space} $, is an equivalence after protuncation if and only if for every truncated object $ K \in C_{<\infty} $, the induced morphism $ f(K) \colon \fromto{X(K)}{Y(K)} $ is an equivalence.
\end{nul}


\begin{exm}\label{exm:twoetaletypes}
	Since truncated objects are hypercomplete, for any $ \infty $-topos $ \XX $, the inclusion $ \incto{\XX^{\hyp}}{\XX} $ of the $ \infty $-topos of hypercomplete objects of $ \XX $ induces an equivalence
	\begin{equation*}
		\equivto{\tau_{<\infty}\Pi_{\infty}(\XX^{\hyp})}{\tau_{<\infty}\Pi_{\infty}(\XX)}
	\end{equation*}
	on protruncated shapes.
	In light of \Cref{exm:etalehomotopy}, the shape of the étale $ \infty $-topos of a locally Noetherian scheme $ X $ agrees with the Artin--Mazur--Friedlander étale homotopy type of $ X $ after protruncation.

	For an arbitrary scheme $ X $, we simply refer to the shape $ \Pi_{\infty}(X_{\et}) $ of the étale $ \infty $-topos $ X_{\et} $ of $ X $ as the \textbfit{étale homotopy type} of $ X $.
\end{exm}

\begin{nul}
	Let $ C $ be a presentable $ \infty $-category.
	The essentially unique functor $ \fromto{\Pro(C)}{C} $ that perserves inverse limits and restricts to the identity $ \fromto{C}{C} $ is right adjoint to the Yoneda embedding $ \yo \colon \incto{C}{\Pro(C)} $ \SAG{Example}{A.8.1.7}.
	Hence we have adjunctions
	\begin{equation*}
		\begin{tikzcd}
			C \arrow[r, hooked, shift left, "\yo"] & \Pro(C) \arrow[l, shift left] \arrow[r, "\tau_{<\infty}", shift left] & \Pro(C_{<\infty}) \arrow[l, hooked', shift left] \period
		\end{tikzcd}
	\end{equation*}
	If Postnikov towers converge in $ C $, i.e., $ C $ is a Postnikov complete presentable $ \infty $-category \SAG{Definition}{A.7.2.1}, then the composite right adjoint is also fully faithful:
\end{nul}

\begin{lem}\label{lem:protruncff}
	Let $ C $ be a Postnikov complete presentable $ \infty $-category (e.g., a Postnikov complete $ \infty $-topos).
	Then the protruncation functor
	\begin{equation*}
		\tau_{<\infty} \colon \fromto{C}{\Pro(C_{<\infty})} 
	\end{equation*}
	is fully faithful.
	Moreover, the essential image of $ \tau_{<\infty} \colon \incto{C}{\Pro(C_{<\infty})} $ is the full subcategory spanned by those protruncated objects $ X $ such that for each integer $ n \geq -2 $, the pro-$ n $-truncation $ \tau_{\leq n}(X) \in \Pro(C_{\leq n}) $ is a constant pro-object.
\end{lem}


\begin{nul}
	Composing the fully faithful functor $ \tau_{<\infty} \colon \incto{\Space}{\Pro(\Space_{<\infty})} $ with the inclusion $ \incto{\Pro(\Space_{<\infty})}{\Pro(\Space)} $ gives another embedding of spaces into prospaces: for a space $ K $, the natural morphism of prospaces $ \fromto{j(K)}{\tau_{<\infty}(K)} $ is an equivalence if and only if $ K $ is truncated.
	Unlike the Yoneda embedding, the functor $ \tau_{<\infty} \colon \incto{\Space}{\Pro(\Space)} $ is neither a left nor a right adjoint.
\end{nul}


\section{Limits \& the protruncated shape}\label{sec:work}

The shape does not preseve inverse limits, even of bounded coherent $ \infty $-topoi.
In this section we prove that, nevertheless, the \textit{protruncated} shape preserves inverse limits of bounded coherent $ \infty $-topoi.
Our main theorem (\Cref{thm:mainthm}) is an easy consequence.


\begin{ntn}
	Write $ \Topbc \subset \Top_{\infty} $ for the subcategory of \textit{bounded coherent} $ \infty $-topoi and \textit{coherent} geometric morphisms \cites[Definitions \SAGthmlink{A.2.0.12} \& \SAGthmlink{A.7.1.2}]{SAG}[Definition 5.28]{exodromy}.
\end{ntn}

\begin{prp}\label{prop:protruncshapeinverselim}
	The protruncated shape 
	\begin{equation*}
		\tau_{<\infty}\Pi_{\infty} \colon \fromto{\Topbc}{\Pro(\Space_{<\infty})}
	\end{equation*}
	preserves inverse limits.
\end{prp}

\begin{proof}
	Let $ \XX \colon \fromto{I}{\Topbc} $ be an inverse system of bounded coherent $ \infty $-topoi and coherent geometric morphisms.
	For each $ i \in I $, the forgetful functor $ \fromto{I_{/i}}{I} $ is limit-cofinal \cite[\HTTthm{Example}{5.4.5.9} \& \HTTthm{Lemma}{5.4.5.12}]{HTT}, so we may without loss of generality assume that $ I $ admits a terminal object $ 1 $.
	For each $ i \in I $, write 
	\begin{equation*}
		\pi_{i,\ast} \colon \fromto{\textstyle\lim_{j \in I} \XX_{j}}{\XX_i}
	\end{equation*}
	for the projection, $ \Gamma_{i,\ast} \coloneq \Gamma_{\XX_i,\ast} $, and $ f_{i,\ast} \colon \fromto{\XX_i}{\XX_1} $ for the geometric morphism induced by the essentially unique morphism $ \fromto{i}{1} $ in $ I $.
	Write $ \Gammalowerstar \colon \fromto{\lim_{j \in I} \XX_{j}}{\Space} $ for the global sections geometric morphism.

	We want to show that the natural morphism
	\begin{equation*}
		\fromto{\colim_{i \in I^{\op}} \Gamma_{i,\ast} \Gammaupperstar_i}{\Gammalowerstar \Gammaupperstar}
	\end{equation*}
	in $ \Fun(\Space,\Space) $ is an equivalence when restricted to truncated spaces \cref{nul:protrun}.
	By \cite[Lemma 8.11]{exodromy} the natural morphism 
	\begin{equation*}
		\fromto{\colim_{i \in I^{\op}} f_{i,\ast} \fupperstar_i}{\pi_{1,\ast} \piupperstar_1}
	\end{equation*}
	is an equivalence in $ \Fun(\XX_1,\XX_1) $.
	Since $ \XX_{1} $ is bounded coherent, the global sections functor $ \Gamma_{1,\ast} \colon \fromto{\XX_1}{\Space} $ preserves filtered colimits of uniformly truncated objects \cites[\SAGthm{Proposition}{A.2.3.1}]{SAG}[Corollary 5.55]{exodromy}. 
	Thus for any truncated space $ K $ we see that 
	\begin{align*}
		\colim_{i \in I^{\op}} \Gamma_{i,\ast} \Gammaupperstar_i(K) &\equivalent \colim_{i \in I^{\op}} \Gamma_{1,\ast} f_{i,\ast} \fupperstar_i \Gammaupperstar_{1}(K) \\ 
		&\equivalence \Gamma_{1,\ast}\paren{\colim_{i \in I^{\op}} f_{i,\ast} \fupperstar_i \Gammaupperstar_{1}(K)} \\
		&\equivalent \Gamma_{1,\ast} \of \paren{\colim_{i \in I^{\op}} f_{i,\ast} \fupperstar_i} \of \Gammaupperstar_{1}(K) \\ 
		&\equivalence \Gamma_{1,\ast} \of \pi_{1,\ast}\piupperstar_1 \of \Gammaupperstar_{1}(K) \\
		&\equivalent \Gammalowerstar \Gammaupperstar(K) \period \qedhere
	\end{align*}
\end{proof}


\subsection{Proof of the Main Theorem}

We now prove the main result of this note.
Recall that we write
\begin{equation*}
	\widetilde{(-)} \colon \equivto{\Pro(\Str_{\pi})}{\StrTopspec_{\infty}}
\end{equation*}
for the equivalence of $ \infty $-categories of \cite[Theorem 10.10]{exodromy}.

\begin{lem}\label{lem:easyequiv}
	The square
	\begin{equation*}
		\begin{tikzcd}
			\Str_{\pi} \arrow[r, hooked, "\widetilde{(-)}"] \arrow[d, "\invert"'] & \StrTopspec_{\infty} \arrow[d, "\Pi_{\infty}"] \\ 
			\Space \arrow[r, hooked, "\yo"'] & \Pro(\Space) 
		\end{tikzcd}
	\end{equation*}
	commutes.
\end{lem}

\begin{proof}
	By the definition of the equivalence $ \equivto{\Pro(\Str_{\pi})}{\StrTopspec_{\infty}} $ of \cite[Theorem 10.10]{exodromy},
	the following square commutes
	\begin{equation*}
		\begin{tikzcd}
			\Str_{\pi} \arrow[r, hooked, "\widetilde{(-)}"] \arrow[d] & \StrTopspec_{\infty} \arrow[d] \\ 
			\Cat_{\infty} \arrow[r, "{\Fun(-,\Space)}"'] & \Top_{\infty} \kern0.5em \comma
		\end{tikzcd}
	\end{equation*}
	where the vertical functors forget stratifications.
	Combining this with \Cref{exm:shapeofpresheaf} proves the claim.
\end{proof}

\begin{nul}
	Since the extension of $ \invert \colon \fromto{\Str_{\pi}}{\Space} $ to pro-objects preserves inverse limits, \Cref{lem:easyequiv} shows that we have a morphism of prospaces 
	\begin{equation*}
		\theta_C \colon \fromto{\Pi_{\infty}(\widetilde{C})}{\invert(C)}
	\end{equation*}
	natural in $ C \in \Pro(\Str_{\pi}) $.
\end{nul}

\begin{thm}\label{thm:mainthm}
	For any profinite stratified space $ C $, the natural map
	\begin{equation*}
		\tau_{<\infty} \theta_C \colon \fromto{\tau_{<\infty}\Pi_{\infty}(\widetilde{C})}{\tau_{<\infty}\invert(C)}
	\end{equation*}
	of protruncated spaces is an equivalence.
\end{thm}

\begin{proof}
	Since the forgetful functor $ \fromto{\StrTopspec_{\infty}}{\Topbc} $ preserves inverse limits, \Cref{prop:protruncshapeinverselim} implies that the protruncated shape $ \tau_{<\infty} \Pi_{\infty} \colon \fromto{\StrTopspec_{\infty}}{\Pro(\Space_{<\infty})} $ preserves inverse limits.
	Both $ \tau_{<\infty} $ and $ \invert $ preserve inverse limits, hence their composite $ \tau_{<\infty} \invert \colon \fromto{\Pro(\Str_{\pi})}{\Pro(\Space_{<\infty})} $ preserves inverse limits.
	The claim now follows from the fact that $ \theta_C $ is an equivalence for $ C \in \Str_{\pi} $ (\Cref{lem:easyequiv}) and the universal property of the $ \infty $-category $ \Pro(\Str_{\pi}) $ of profinite stratified spaces.
\end{proof}

\begin{nul}
	Note that \Cref{thm:mainAG} from the introduction is immediate from \Cref{thm:mainthm}, \cite[Construction 13.5]{exodromy}, and the definition of the étale homotopy type in terms of shape theory (\Cref{exm:etalehomotopy,exm:twoetaletypes}).
\end{nul}



\DeclareFieldFormat{labelnumberwidth}{#1}
\printbibliography[keyword=alph]
\addcontentsline{toc}{section}{References} 
\DeclareFieldFormat{labelnumberwidth}{{#1\adddot\midsentence}}
\printbibliography[heading=none, notkeyword=alph]

\end{document}